\documentclass[oneside,english]{amsart}
\usepackage[T1]{fontenc}
\usepackage[latin9]{inputenc}
\usepackage{amsthm}
\usepackage{amstext}
\usepackage{amssymb}
\usepackage{esint}

\makeatletter

\numberwithin{equation}{section}
\numberwithin{figure}{section}
\theoremstyle{plain}
\newtheorem{thm}{\protect\theoremname}[section]
  \theoremstyle{plain}
  \newtheorem{prop}[thm]{\protect\propositionname}
  \theoremstyle{plain}
  \newtheorem{cor}[thm]{\protect\corollaryname}
  \theoremstyle{plain}
  \newtheorem{exa}[thm]{\protect\examplename}
  \theoremstyle{plain}
  \newtheorem{lem}[thm]{\protect\lemmaname}
  \theoremstyle{definition}
  \newtheorem{defn}[thm]{\protect\definitionname}
	\theoremstyle{remark}
  \newtheorem{rem}[thm]{\protect\remarkname}

\makeatother

\usepackage{babel}
  \providecommand{\corollaryname}{Corollary}
  \providecommand{\definitionname}{Definition}
  \providecommand{\lemmaname}{Lemma}
  \providecommand{\propositionname}{Proposition}
  \providecommand{\examplename}{Example}
  \providecommand{\theoremname}{Theorem}
	\providecommand{\remarkname}{Remark}

\DeclareMathOperator{\cp}{cap}

\DeclareMathOperator{\dist}{dist}

\begin{document}

\title{Boundary Values of Functions of Dirichlet Spaces $L^1_2$ on Capacitary Boundaries}

\author{V.~Gol'dshtein and A.~Ukhlov}

\begin{abstract}
We prove that any weakly differentiable function with square integrable gradient can be extended to a capacitary boundary of any simply connected plane domain   $\Omega\ne\mathbb R^2$ except a set of a conformal capacity zero. For locally connected at boundary points domains  the capacitary boundary coincides with the Euclidean one.  A concept of  a capacitary boundary was proposed by V.~Gol'dshtein and S.~K.~Vodop'yanov in 1978 for a study of boundary behavior of quasi-conformal homeomorphisms. We prove in details the main properties of the capacitary boundary.  An abstract version of the extension property for more general classes of plane domains is discussed also. 

\end{abstract}
\maketitle

\section{Introduction }

Let  $\Omega$ be a domain in $ \mathbb R^2$.  We consider a Dirichlet space (a uniform Sobolev space) $L^1_2(\Omega)$  of locally integrable functions with the square integrable weak gradient $\nabla u \in L_2(\Omega)$ equipped with the seminorm 
$$
\|u | L^1_2(\Omega)\|= \|\nabla u | L_2(\Omega)\|.
$$
 
The paper is devoted to study of the boundary behavior for functions $u \in  L^1_2(\Omega)$.  

By the standard definition functions of $L^1_2(\Omega)$ are defined only up to a set of measure zero, but they can be redefined quasieverywhere i.~e. up to a set of conformal capacity zero. Indeed, every function $u\in L^1_2(\Omega)$ has a unique quasicontinuous representation $\tilde{u}\in L^1_2(\Omega)$. A function $\tilde{u}$ is termed quasicontinuous if for any $\varepsilon >0$ there is an open  set $U_{\varepsilon}$ such that the conformal capacity of $U_{\varepsilon}$ is less then $\varepsilon$ and the function  $\tilde{u}$ is continuous on the set $\Omega\setminus U_{\varepsilon}$ (see, for example \cite{HKM,Maz}). The concept of quasicontinuity can be obviously extended to the closure $\overline{\Omega}$ of $\Omega$.

In this paper we deals with quasicontinuous representations of functions $u\in L^1_2(\Omega)$.

One of main results of the paper is:

\begin{thm} Let $\Omega\subset\mathbb R^2$, $\Omega\ne\mathbb R^2$, be a simply connected  domain which is locally
connected at any boundary point $x\in\partial\Omega$. Then for any function $u\in L_{2}^{1}(\Omega)$
there exists a quasicontinuous function $\widetilde{u}:\overline{\Omega} \to \mathbb R$ such that $\tilde{u}|_{\Omega}=u$. 
\end{thm}

\begin{rem} The quasicontinuous function   $\widetilde{u}:\overline{\Omega} \to \mathbb R$  is defined at any point of $\partial \Omega$ except a set of conformal capacity zero (i.e. quasieverywhere). 
\end{rem}

The main ingredient of our method is a well-known concept of the conformal capacity and a less known concept of the conformal capacitary boundary introduced by V.~Go\-l'd\-sh\-tein and S.~K.~Vodop'yanov \cite{GV} for quasiconformal homeomorphisms. In the plane case "points" of the conformal capacitary boundary coincide with the Caratheodory prime ends.

Main properties of the space $L^1_2(\mathbb D)$ where $\mathbb D\subset\mathbb R^2$ is the unit disc are well known. Dirichlet spaces  $L^1_2(\Omega)$ are conformal invariants. Therefore the Riemann Mapping Theorem permits us to transfer necessary information about boundary behavior of spaces $L^1_2(\Omega$ from $L^1_2(\mathbb D)$ in the case of simply connected domains $\Omega$.
 
More precisely, we extend the concept of quasicontinuity to a "capacitary" completion of a domain $\Omega$. We construct a conformal capacitary boundary as a completion $\big\{ \widetilde{\Omega_{\rho}},\rho\big\}$ of a metric space $\big\{{\Omega_{\rho}},\rho\big\}$ for a conformal capacitary metric $\rho$ (see section 1). Roughly speaking, an "ideal" capacitary boundary point is a boundary continuum of the conformal capacity zero.

Our method allow us to treat the general case of simply connected plane domains $\Omega\subset\mathbb R^2$. We prove that any function $u\in L^1_2(\Omega)$ has a quasicontinuous extension onto the conformal capacitary boundary $H_{\rho}=\widetilde{\Omega}_{\rho}\setminus{\Omega}_{\rho}$. The main result is: 

\begin{thm} 
Let $\Omega\subset\mathbb R^2$ be a simply connected domain, $\Omega\ne\mathbb R^2$. Then for any function $u\in L_{2}^{1}(\Omega)$
there exists a quasicontinuous function $\widetilde{u}:\widetilde{\Omega}_{\rho}\to \mathbb R$
such that $\tilde{u}|_{\Omega}=u$. 
\end{thm}

\begin{rem} A concept of the conformal capacitary metric $\rho$ and the conformal capacitary boundary was proposed in \cite{GV}. By a quasi-invariance of the conformal capacity under (quasi)conformal homeomorphisms any such homeomorphism $\varphi: \Omega\to\Omega'$ is a bi-Lipschitz homeomorphism $\varphi: (\Omega, \rho)\to(\Omega', \rho)$ for corresponding conformal metrics and can be extended to a homeomorphism $\tilde{\varphi}: (\tilde{\Omega}, \rho)\to(\tilde{\Omega'}, \rho)$ of the capacitary completions \cite{GV}. Recall that the paper \cite{GV}  is a short note and contains only sketches of proofs.
\end{rem}

There is the vast literature concerning of "ideal" boundaries of plane domains in the context of conformal homeomorphisms. We discuss few such concepts in the last section. 

The paper is organized as follows:

Main properties of the conformal capacitary metric are proved in Section 2.  The focus is on the local properties of the  metric at boundary points and its dependence on the local topological properties of the boundary. In section 3 we discuss an analog of the Luzin property for the capacitary metric. In section 4 we discuss a sufficient condition for existence of an extension of functions $u\in L_{2}^{1}(\Omega)$ to the capacitary boundary. We call this condition as a strong Luzin property for the capacitary metric. We prove this condition for comparatively large classes of domains that include extension domains for  $ L_{2}^{1}(\Omega)$. 
In Section 5 we apply the abstract construction of Section 4 to simply connected plane domains and we prove main results about extension of functions $u\in L_{2}^{1}(\Omega)$ to the capacitary boundary. 

{\it In terminology of the theory of Sobolev spaces we solved the classical trace problem for $ L_{2}^{1}(\Omega)$ in simply connected plane domains.}

\begin{rem} The classical trace problem for Sobolev spaces is of essential interest, mainly due to its important applications to boundary-value problems for partial differential equations. Boundary value problems can be specified with the help of traces to $\partial\Omega$ of Sobolev functions.

There is an extensive literature devoted to the trace problem of Sobolev spaces.
Among the multitude of results we mention the monographs of P.~Grisvard \cite{Gri}, J.~L.~Lions and E.~Magenes
\cite{LM}, V.~G.~Maz'ya and S.~Poborchi \cite{Maz}, \cite{MazP}, and the papers \cite{AS}, \cite{Be}, \cite{Ga}, \cite{J1}, 
\cite{J2}, \cite{JW}, \cite{MazP1}, \cite{MazP2}, \cite{MazP3}, \cite{MazPN},   \cite{Nik}, \cite{P1}, \cite{Vas}, \cite{Yak1}, 
\cite{Yak2}.

For smooth domains the traces of Sobolev functions are Besov spaces. In the case of Lipschitz domains the traces can be described also in terms of Besov spaces. For arbitrary non Lipschitz domain the trace problem is open.  For cusp type singularities a description of traces can be found in \cite{GVas} in terms of weighted Sobolev spaces.
\end{rem}

\section{ Conformal Capacitary Metric }

Let $\Omega$ be a plane domain and $F_{0}$, $F_{1}$ two disjoint compact subset of $\Omega$. We call the triple $E=(F_{0},F_{1};\Omega)$ a condenser.

The value \[
\cp(E)=\cp(F_{0},F_{1};\Omega)=\inf\int\limits _{\Omega}|\nabla v|^{2}~dx,
\]
 where the infimum is taken over all nonnegative functions $v\in C(\Omega)\cap L_{2}^{1}(\Omega)$,
such that $v=0$ in a neighborhood of the set $F_{0}$, and $v\geq1$
in a neighborhood of the set $F_{1}$, is called the conformal capacity
of the condenser $E=(F_{0},F_{1};\Omega)$. 

For finite values of capacity $0<\cp(F_{0},F_{1};\Omega)<+\infty$
there exists a unique function $u_{0}$ (an extremal function) such
that: \[
\cp(F_{0},F_{1};\Omega)=\int\limits _{\Omega}|\nabla u_{0}|^{2}~dx.\]
 An extremal function is continuous in $\Omega$, monotone in the
domain $\Omega\setminus(F_{0}\cup F_{1})$, equal to zero on $F_{0}$
and is equal to one on $F_{1}$ \cite{HKM,VGR}. 

\begin{defn}
A homeomorphism $\varphi:\Omega\to\Omega'$
between plane domains is called $K$-quasiconformal if it preserves
orientation, belongs to the Sobolev class $L^1_{2,loc}(\Omega)$
and the distortion inequality
$$
\max\limits_{|\xi|=1}|D\varphi(x)\cdot\xi|\leq K\min\limits_{|\xi|=1}|D\varphi(x)\cdot\xi|
$$
holds for almost all $x\in\Omega$.
\end{defn}

Infinitesimally, quasiconformal homeomorphisms carry circles to ellipses
with eccentricity uniformly bounded by $K$. If $K=1$ we recover
conformal homeomorphisms, while for $K>1$ plane quasiconformal mappings need
not be smooth. The theory of quasiconformal mappings can be find, for example, in \cite{Va}.

It is well known that the conformal capacity is quasi-invariant under action of plane quasiconformal homeomorphisms. 

\subsection{Definition of the conformal capacitary metrics}

A connected closed (with respect to $\Omega$) set is called a
continuum.  Fix a continuum $F$ in the domain $\Omega\subset\mathbb{R}^{2}$
and a compact domain $V$ such that $F\subset V\subset\overline{V}\subset\Omega$,
and the boundary $\partial V$ is an image of the unit circle $S(0,1)$ under some
quasiconformal 
homeomorphism of $\mathbb{R}^{2}$.

\begin{defn}
Choose arbitrarily points $x,y\in\Omega\subset\mathbb R^n$ and
joint $x,y$ by a rectifiable curve $l(x,y)$. Define the conformal capacitary distance
between $x$ and $y$ in $\Omega$ with respect to pair $(F,V)$ as the following quantity
 \[
\rho_{(F,V)}(x,y)=\inf\limits _{l(x,y)}\{\cp^{\frac{1}{2}}(F,l(x,y)\setminus V;\Omega)+\cp^{\frac{1}{2}}(\partial\Omega,l(x,y)\cap V;\Omega)\}\]
 where the infimum is taken over all curves $l(x,y)$ satisfying the
above conditions. 
\end{defn}

This definition firstly was introduced in \cite{GV}, where was claimed that the distance $\rho_{(F,V)}(x,y)$ is a metric in $\Omega$ which is  quasi-invariant under quasiconformal homeomorphisms and is invariant under conformal ones.

Denote by $\big\{ \widetilde{\Omega},\rho_{(F,V)}\big\} $ the
standard completion of the metric space $\big\{ \Omega,\rho_{(F,V)}\big\} $
and by $H_{\rho}$ the set $\big\{ \widetilde{\Omega},\rho_{(F,V)}\big\}\setminus \left\{ \Omega,\rho_{(F,V)}\right\}$.
We call $H_{\rho}$ a conformal capacitary boundary of $\Omega$. It will be
proved that the topology of $H_{\rho}$ does not depends on choice of
a pair $(F,V)$. Moreover two conformal capacitary metrics are equivalent
for any different choice of pairs $(F_{1},V_{1})$ and $(F_{2},V_{2})$.
This is a justification of the notation $H_{\rho}$ for the conformal capacitary
boundary.

\subsection{Equivalence of different conformal capacitary distances}

We start from an important technical observation:

\begin{lem} 
\label{lem:QuasiEst}
Let $\Omega$ be a domain in $\mathbb R^2$ and $F_{1}$ be a compact in $\Omega$. 
Then there exists a constant $0<K<\infty$ such that
\[
\frac{1}{K}\cp(F_{02},F_1;\Omega)\leq \cp(F_{01},F_1;\Omega)\leq {K}\cp(F_{02},F_1;\Omega)
\]
for every compacts $F_{01}\subset\Omega$, $F_{02}\subset\Omega$ such that compacts $F_{01}$, $F_{02}$, $F_{1}$ are mutually disjoint.

\end{lem} 

\begin{proof} 
Let $U_{1}\supset F_{01}$ be a $\varepsilon$-neighborhood
of a set $F_{01}$ in $\Omega$ and $U_{2}\supset F_{02}$ be a $\varepsilon$-neighborhood
of the set $F_{02}$ in $\Omega$ such that $\overline{U_1}\subset\Omega$ and $\overline{U_2}\subset\Omega$. 
Because the set $F_{01}$ is a compact set there exist a
finite covering $\left\{ B_{i}\right\} _{i=1,...,N}$ of $F_{01}$ by
balls $B_{i}\in\Omega$ of radius $\varepsilon$. It means that $\Omega\supset\bigcup_{i=1}^{N}B_{i}\supset F_{1}$.

Then \[
\cp(F_{01},F_1;\Omega)\leq\cp(\bigcup_{i=1}^{N}\overline{B_{i}},F_1;\Omega).\]

For every ball $B_i$, $i=2,..N$, we can construct a bi-Lipschitz homeomorphism $\psi_i$ of $\Omega$ onto itself that maps 
$B_i$ onto $B_1$.
Using quasi-invariance of the conformal capacity under bi-Lipschitz homeomorphisms we have
\[
\cp(F_{01},F_1;\Omega)\leq\cp(\bigcup_{i=1}^{N}\overline{B_{i}},F_1;\Omega)\leq C_{1}\cdot\cp(\overline{B_1},F_1;\Omega)\]
 where a constant $C_{1}$ depends only on the multiplicity of the
covering and $F_{1}$. 

Applying the same construction to $F_{02}$ we can construct a finite covering $\left\{ \tilde{B}_{i}\right\} _{i=1}^{N}$ of $F_{02}$ by
balls $\tilde{B}_{i}\in\Omega$ of radius $\varepsilon$ such that $\Omega\supset\bigcup_{i=1}^{N}\tilde{B}_{i}\supset F_{02}$. Then
 
\[\cp(F_{02},F_1;\Omega)\leq C_{2}\cdot\cp(\overline{\tilde{B}_1},F_1;\Omega)\]
where a constant $C_{2}$ depends only on the multiplicity of the
covering of $F_{2}$ and $F_{2}$ itself. 

To combine both estimate we construct a bi-Lipschitz homeomorphism
$\varphi$ of $\Omega$ that maps the ball $B_1\subset U_{1}$
onto a ball $\tilde{B}_1\subset U_{2}$. Using quasi
invariance of the conformal capacity under bi-Lipschitz homeomorphisms
we obtain
\[
\cp(\overline{B_1},F_1;\Omega)\leq M\cp(\overline{\tilde{B}_1},F_1;\Omega)\leq M\cp(\overline{U_{2}},F_1;\Omega)
\]
where $M$ is the Lipschitz constant of $\varphi$. Hence
\[
\cp_{p}(F_{01},F_1;\Omega)\leq C_{1}M\cp_{p}(\overline{U_{2}},F_1;\Omega).
\]
Therefore 
\begin{equation}
\cp(F_{01},F_1;\Omega)\leq\lim_{\varepsilon\to 0}K\cp(\overline{U_{2}},F_1;\Omega)=K\cp(F_{02},F_1;\Omega).\label{eq:Pcap}
\end{equation}

Using the same construction and the inverse bi-Lipschitz homeomorphisms $\varphi^{-1}$ we obtain

\begin{equation}
\cp(F_{02},F_1;\Omega)\leq\lim_{\varepsilon\to 0}C_{2}{M}\cp(\overline{U_{1}},F_1;\Omega)=C_{2}{M}\cp(F_{01},F_1;\Omega).\label{eq:Pcap2}
\end{equation}
Using inequalities \ref{eq:Pcap} and \ref{eq:Pcap2} we obtain
\[
\frac{1}{K}\cp(F_{02},F_1;\Omega)\leq \cp(F_{01},F_1;\Omega)\leq {K}\cp(F_{02},F_1;\Omega). 
\]
\end{proof}

Using this lemma we prove

\begin{thm}
\label{thm:QuasiInvMet}
Let $\Omega$ be a domain in $\mathbb R^2$. Suppose that $\rho_{(F_{1},V_{1})}$ and $\rho_{(F_{2},V_{2})}$ are two conformal capacitary distances  on $\Omega$.
Then there exists a constants $0<K<\infty$ such that
\[
\frac{1}{K}\rho_{(F_{2},V_{2})}(x,y)\leq\rho_{(F_{1},V_{1})}(x,y)\leq K\rho_{(F_{1},V_{1})}(x,y)\,\,\,\text{for any}\,\,\,x,y\in \Omega.
\]
\end{thm}

\begin{proof} Consider two continuums $F_{1}$ and $F_{2}$ in the
domain $\Omega$. 
Without loss of generality we can suppose that any admissible for
the conformal capacitary metric curve $l(x,y)$ does not intersect $V_{1}$ and $V_{2}$.
Then by Lemma \ref{lem:QuasiEst}
\[
\frac{1}{K}\cp(F_{2},l(x,y);\Omega)\leq \cp(F_{1},l(x,y);\Omega)\leq {K}\cp(F_{2},l(x,y);\Omega) 
\]
for any admissible for
the conformal capacitary metric curve $l(x,y)$. Hence
\[
\frac{1}{K}\rho_{(F_{2},V_{2})}(x,y)\leq\rho_{(F_{1},V_{1})}(x,y)\leq K\rho_{(F_{1},V_{1})}(x,y)\]
for any $x,y\in \Omega$.
\end{proof}

\subsection{Conformal capacitary distance is a metric}

The sketch of the proof can be found in \cite{GV}. For readers convenience we prove this fact in details.

\begin{lem} 
\label{lem:CapContin}
Let $\Omega$ be a domain in $\mathbb R^2$ and $F$ be a compact subdomain of $\Omega$. Suppose that
$x\in\Omega\setminus\overline{F}$, $B(x,2r)\subset\Omega\setminus\overline{F}$, $r>0$ and curve $\gamma$ joints $x$ and $S(x,2r)$. Then  
$\cp(F,\gamma;\Omega)>c(r)>0$.
\end{lem}

\begin{proof} Choose a continuum $F_r$, such that $F_r\cap\gamma=\emptyset$, which connects spheres $S(x,r)$ and $S(x,2r)$.
By Lemma~\ref{lem:QuasiEst}
$$
\cp(F,\gamma;\Omega)\geq Q\cp(F_r,\gamma;\Omega).
$$
Using the properties of capacity and Proposition~4.6 from \cite{GGR} we obtain
$$
\cp(F_r,\gamma;\Omega)\geq \cp(F_r,\gamma\cap\{z:r\leq |x-z|\leq 2r \};\Omega)\geq c(r)>0.
$$
Hence
$\cp(F,\gamma;\Omega)\geq c(r)>0$.
\end{proof} 

\begin{thm} 
\label{thm:ProofMetric}
Let $\Omega$ be a domain in $\mathbb R^2$. Then the $\rho_{(F,V)}(x,y)$
is a metric in $\Omega$.
\end{thm} 

\begin{proof} By standard properties of the conformal capacity \cite{Maz}  $\rho_{(F,V)}(x,x)=0$.
 
Let $\rho_{(F,V)}(x,y)=0$. Assume by contradiction that $x\ne y$ and denote the Euclidean distance $\dist(x,y)$ as $r$. Then there exists a sequence of curves $\{l_{k}(x,y)\}$,
$k=1,2,...$, such that 
$$\cp(F,l_{k}(x,y)\setminus V;\Omega)\rightarrow 0\,\,\,\text{while}\,\,\, k\rightarrow\infty
$$ 
and 
$$\cp(\partial\Omega,l_{k}(x,y)\cap V;\Omega)\rightarrow 0\,\,\,\text{while}\,\,\,k\rightarrow\infty.
$$ 
Hence a sequence of extremal
functions $u_{k}\in L_{2}^{1}(\Omega)$ for the capacities 
$$
\cp (F,l_{k}(x,y)\setminus V;\Omega)
$$ 
tends to zero in the space $L_{2}^{1}(\Omega)$ and a sequence of extremal
functions $v_{k}\in L_{2}^{1}(\Omega)$ for the capacities 
$$
\cp(\partial\Omega,l_{k}(x,y)\cap V;\Omega)
$$
tends to zero in the space $L_{2}^{1}(\Omega)$. 

So, the sequence $u_{k}$ tends to zero except a set of the conformal capacity zero and the sequence $v_{k}$ tends to zero except for a set of the conformal capacity zero. But by Lemma~\ref{lem:CapContin}
$$
\|u_k| L^1_2(\Omega)\|+\|v_k| L^1_2(\Omega)\|\geq c(r)>0\quad \text{for all} \quad k.
$$
Contradiction. Therefore $x=y$.

By the definition of the conformal capacity we have
\begin{multline}
\cp^{\frac{1}{2}}(F,l(x,y)\setminus V;\Omega)+\cp^{\frac{1}{2}}(\partial\Omega,l(x,y)\cap V;\Omega)\\
=\cp^{\frac{1}{2}}(F,l(y,x)\setminus V;\Omega)+\cp^{\frac{1}{2}}(\partial\Omega,l(y,x)\cap V;\Omega).
\nonumber
\end{multline}
 Hence $\rho_{(F,V)}(x,y)=\rho_{(F,V)}(y,x)$. 

Prove the triangle inequality. Choose arbitrary points $x,y,z\in\Omega$.
By the subadditive property of  the capacity \cite{Maz} we obtain
\begin{multline}
\cp^{\frac{1}{2}}(F,(l(x,z)\setminus V)\cup(l(z,y)\setminus V);\Omega)\\
\leq\cp^{\frac{1}{2}}(F,(l(x,z)\setminus V);\Omega)+\cp^{\frac{1}{2}}(F,(l(z,y)\setminus V);\Omega)
\nonumber
\end{multline}
 and
\begin{multline}
\cp^{\frac{1}{2}}(\partial\Omega,(l(x,z)\cap V)\cup(l(z,y)\cap V);\Omega)\\
\leq\cp^{\frac{1}{2}}(\partial\Omega,(l(x,z)\cap V);\Omega)+\cp^{\frac{1}{2}}(\partial\Omega,(l(z,y)\cap V);\Omega).
\nonumber
\end{multline}

Hence $\rho_{(F,V)}(x,y)\leq\rho_{(F,V)}(x,z)+\rho_{(F,V)}(y,z)$.
Therefore $\rho_{p;(F,V)}$ is a metric.
\end{proof}

\begin{thm}
\label{thm:CoinTopol}
The topology induced by the conformal capacitary metric $\rho_{(F,V)}$ into the domain $\Omega\subset\mathbb R^2$ coincides with the
Euclidean topology.
\end{thm}

\begin{proof} Let $U\subset\Omega$ be an open set with respect
to Euclidean metric. Let $x_{0}\in U$ such that \[
\overline{B(x_{0},3r)}=\overline{\{x\in\Omega:|x-x_{0}|<3r\}}\subset U.\]
Then for every point $y\in\partial B(x_{0},r)$ we have that 
\begin{multline}
\rho_{(F,V)}(x_{0},y)=\\
\inf\limits _{l(x_{0}y)}\{\cp^{\frac{1}{2}}(F,l(x_{0}y)\setminus V;\Omega)+\cp^{\frac{1}{2}}(\partial\Omega,l(x_{0}y)\cap V;\Omega)\}>0,
\nonumber
\end{multline}
since $H^{1}$-Hausdorff measure of the sets $l(x_{0}y)\setminus V$
(or $l(x_{0}y)\cap V$ is positive). Hence $U$ is the open set with
respect to the conformal capacity metric in the domain $\Omega$. The inverse
inclusion can be proved similarly. 
\end{proof}

We complete by the standard manner the metric space $\Omega_{\rho}=\Omega_{\rho_{(F,V)}}$.
In the completion $\tilde{\Omega}_{\rho}$ we define the conformal capacitary
boundary of $\Omega$ as $H_{\rho}=\tilde{\Omega}_{\rho}\setminus{\Omega_{\rho}}$.
It means that a boundary element $h\in H_{\rho}$ is a class of fundamental
(in the metric $\rho$) sequences $\{x_{n}\}_{n=1}^{\infty}$. 
\vskip 0.3cm 

\begin{thm}
\label{thm:IsometryMetric}
Suppose that $\rho_{(F_{1},V_{1})}$ and $\rho_{(F_{2},V_{2})}$ are two conformal capacitary metrics on $\Omega$.
The metric spaces $\left\{ H_{\rho},\rho_{(F_{1},V_{1})}\right\} $
and $\left\{ H_{\rho},\rho_{(F_{2},V_{2})}\right\} $ are quasi-isometric, i.e there exist a constant $0<K<\infty$ such that
\[
\frac{1}{K}\rho_{(F_{2},V_{2})}(x,y)\leq\rho_{(F_{1},V_{1})}(x,y)\leq K\rho_{(F_{1},V_{1})}(x,y)\]
for any $x,y\in H_{p}$. 
\end{thm} 

\begin{proof} Consider two continuums $F_{1}$ and $F_{2}$ in the
domain $\Omega$ and suppose that sequences $\{x_{k}\}$ and $\{y_{k}\}$ are
fundamental sequences in the metric $\rho_{(F_{2},V_{2})}$. If
the sequences $\{x_{k}\}$ and $\{y_{k}\}$ have limit point $x,y\in\Omega$, then
the sequences are fundamental in the metric $\rho_{(F_{1},V_{1})}$
because in the domain $\Omega$ conformal capacitary topologies are equivalent
to the Euclidean topology. Let fundamental sequences $\{x_{k}\in\Omega\}$
and $\{y_{k}\in\Omega\}$ has limit points$\,$$x,y\in\left\{ H_{\rho},\rho_{(F_{2},V_{2})}\right\} $.

By Theorem~\ref{thm:QuasiInvMet} for any $x_k, y_k\in\Omega$ there exist a constant $0<K<\infty$ such that
$$
\frac{1}{K}\rho_{(F_{2},V_{2})}(x_k,y_k)\leq\rho_{(F_{1},V_{1})}(x_k,y_k)\leq K\rho_{(F_{1},V_{1})}(x_k,y_k)
$$
Passing to limit by $k\to\infty$ we conclude the proof.
\end{proof}

\begin{rem} Both metrics are equivalent in completions of $\Omega$.
It can be proved by the same way but it  is more technical. For
our main aims this fact is not important.
\end{rem}
\vskip 0.3cm 

\subsection{Asymptotic behavior of the conformal capacitary metric}

For study of an asymptotic behavior of the conformal capacitary metric we need a few well known estimates of the conformal capacity (see, for example \cite{Maz}). For readers convenience we reproduce simple proofs of these facts adapted to the conformal capacitary metric study:

\begin{lem} Consider the unit disc $\mathbb D(0,1)\subset\mathbb R^2$ and continuums 
$$
F_0=(-1,-\frac{1}{2}]\subset \mathbb D\,\,\,\text{and}\,\,\, F_1=[0,\varepsilon]\subset \mathbb D, \,\,0<\varepsilon<\frac{1}{4}.
$$
Then 
$$
\cp(F_0, F_1; \mathbb R^2 \geq c_1\ln(1+\varepsilon),\,\,\, c_1=const.
$$
\end{lem} 

\begin{proof} Consider the conformal mapping 
$$
\varphi: \mathbb C\to \mathbb C,\,\,\, \varphi(z)=\frac{1}{z}.
$$
Then by the capacity estimates \cite{GR}
\begin{multline}
\cp(F_0, F_1; \mathbb R^2)=\cp(\varphi(F_0), \varphi(F_1); \mathbb R^2)\geq \cp([-2,-1], [\frac{1}{\varepsilon},1+\frac{1}{\varepsilon}]; \mathbb R^2)\\
=c_1\ln{\frac{1+\frac{1}{\varepsilon}}{\frac{1}{\varepsilon}}}=c_1\ln(1+\varepsilon).
\nonumber
\end{multline}
\end{proof}

\begin{lem} Consider the unit disc $\mathbb D(0,1)\subset\mathbb R^2$ and continuums $F_0=\{z: |z|\geq 1\}$ and $F_1=[0,\varepsilon]\subset \mathbb D$, $0<\varepsilon<\frac{1}{4}$. Then 
$$
\cp(F_0, F_1; \mathbb R^2) \leq c_2\biggr(\ln\frac{1}{\varepsilon}\biggl)^{-1},\,\,\, c_2=const.
$$
\end{lem}

\begin{proof} By the capacity estimates \cite{GR}
$$
\cp(F_0, F_1; \mathbb R^2)\leq\cp(F_0, \overline{D(0,\varepsilon)}, \mathbb R^2)=c_2\biggr(\ln\frac{1}{\varepsilon}\biggl)^{-1}.
$$
\end{proof}

From these lemmas immediately follows:

\begin{prop} Let $x=(0,0)$ and $y=(\varepsilon,0)$ are points of the unit disc $\mathbb D(0,1)\subset\mathbb R^2$. Then we have the following asymptotic behavior of the conformal metric in the unit disk $D(0,1)$:
$$
\lim\limits_{x\to y}\frac{\rho(x,y)}{|x-y|}\geq \lim\limits_{\varepsilon\to 0}c_1\frac{\ln(1+\varepsilon)}{\varepsilon}=c_1.
$$
and
$$
\lim\limits_{x\to y}\frac{\rho(x,y)}{|x-y|}\leq \lim\limits_{\varepsilon\to 0}c_2 \frac{\biggr(\ln\frac{1}{\varepsilon}\biggl)^{-1}}{\varepsilon}<+\infty.
$$
\end{prop}

Now we study topological properties of the boundary $H_{\rho}$.

\begin{defn} For arbitrary point $h\in H_{\rho}$ we consider
disks $D(h,\varepsilon$), $\varepsilon>0$, defined in terms of the conformal capacitary
metric $\rho_{(F,V)}$. 

Call the set \[
s_{h}=\bigcap_{\varepsilon>0}\overline{D(h,\varepsilon)\cap\Omega}\subset\overline{\mathbb{R}^{2}}\]
the  realization (impression) of a boundary element $h\in H_{\rho}$. 
\end{defn}

Recall that a domain $\Omega$ is called locally connected at a point $z_0\in \partial \Omega$ if $z_0$ has arbitrarily small connected
neighborhoods in $\Omega$. 
By C.~Caratheodory \cite{Cr} the domain $\Omega$ is locally connected at every boundary point if and only if every prime end  has trivial realization.

\begin{lem} 
\label{lem:OnePoint}Let a realization $s_{h}$ of a boundary
element $h\in H_{\rho}$ is one-point. Then for every sequence $\{x_{m}\in\Omega\}$
from  $\rho_{(F,V)}(x_{m},h)\rightarrow0$
 follows that $|x_{m}-s_{h}|\rightarrow0$ (while $m\rightarrow\infty$). 
\end{lem} 

\begin{proof} Suppose that $\rho_{(F,V)}(x_{m},h)\rightarrow0$ while $m\rightarrow\infty$.
Because the realization $s_{h}$ of a boundary element $h\in H_{\rho}$ is a point, then
$$
diam\left(\overline{D(h,\varepsilon)\cap\Omega}\right)=\sup\limits_{x,y\in \overline{D(h,\varepsilon)\cap\Omega}}|x-y|
\to 0\,\,\,\text{while}\,\,\,\varepsilon\to 0.
$$
The sequence $\{x_{m}\}$ belongs to a boundary element $h\in H_{\rho}$ and so
$$
|x_m-x_n|\to 0\,\,\,\text{while}\,\,\,m,n\to\infty.
$$
Hence, the sequence $\{x_n\}$ is a sequence Cauchy in the Euclidean metric, and we have that $|x_{m}-s_{h}|\rightarrow0$
while $m\rightarrow\infty$. 
\end{proof}
 
\begin{lem} 
\label{lem:LocConnected}
Let a domain $\Omega$ is locally
connected at a point $x\in\partial\Omega$ and $x\in s_h$ for some and $h\in H_{\rho}$. Then for every sequence $\{x_{m}\in\Omega\}$ such that $|x_{m}-x|\rightarrow 0$
we have $\rho_{(F,V)}(x_{m},h)\rightarrow 0$ (while $m\rightarrow\infty$). 
\end{lem} 

\begin{proof} Since the domain $\Omega$ be locally connected at
the point $x\in\partial\Omega$ then any two points $x_{k},x_{m}$
from the sequence $\{x_{n}\}$ can be connected by a geodesic path
$l(x_{k},x_{m})$ such that its length tends to zero for $k,m\to\infty$.
Without loss of generality we can suppose that $l(x_{k},x_{m})\cap V=\emptyset$.
Hence $\cp_{p}(F_{1},l(x_{k},x_{m});\Omega)$ tends to zero for $k,m\to\infty$
and therefore\[
\lim\limits _{n\rightarrow\infty}\rho_{(F,V)}(x_{n},h)=0.\]
\end{proof}

From these lemmas follows:

\begin{thm} 
\label{LocCon}
Let a domain $\Omega$ is locally
connected at any point $x\in\partial\Omega$. Then the identical mapping $i:\Omega\to\Omega$ can be extend to a homeomorphism
$\tilde{i_\rho}: \overline{\Omega}\to \tilde{\Omega}_{\rho}$ if and only if  all realizations $s_{h}$ of boundary
elements $h\in H_{\rho}$ are one-points. 
\end{thm}

\begin{proof} Suppose that an identical mapping $i:\Omega\to\Omega$ can be extended to a homeomorphism
$\tilde{i_\rho}: \overline{\Omega}\to \tilde{\Omega}_{\rho}$. Then every boundary element $h\in H_{\rho}$ coincides with a point $x\in \partial\Omega$ and so has an one-point realization. 

Inversely, let all realizations $s_{h}$ of a boundary elements $h\in H_{\rho}$ are one-points. 
Then extending an identical mapping $i:\Omega\to\Omega$ to the mapping $\tilde{i_{\rho}}:\tilde{\Omega}_{\rho}\to\overline{\Omega}$ by the rule $\tilde{i_{\rho}}(h)=s_h$ we obtain a one-to-one correspondence $\tilde{i_{\rho}}: \tilde{\Omega}_{\rho}\to\overline{\Omega}$.
Let us check continuity of $\tilde{i_{\rho}}$ and ${\tilde{i_p}}^{-1}$. 

Suppose that $x_k\to x$ in $\overline{\Omega}$. Because all realizations $s_{h}$ of boundary elements $h\in H_{\rho}$ are one-points and $x\in s_h$ then by Lemma~\ref{lem:LocConnected} follows $\rho_{(F,V)}(x_{m},h)\rightarrow 0$, while $m\rightarrow\infty$. We prove that $\tilde{i_{\rho}}$ is continuous.

Suppose $h_k\to h_0$ in $\tilde{\Omega}_{\rho}$. Because $\Omega$ is locally connected then realizations of $h_k$ and $h_0$ are one point sets and we can identify $h_k$ and $h_0$ with their realizations. By Lemma~\ref{lem:OnePoint} $h_k\to h_0$ in $\Omega$. Therefore ${\tilde{i_{\rho}}}^{-1}$ is also continuous. 
Therefore $\tilde{i_{\rho}}$ is a homeomorphism.
\end{proof}

\subsection{Conformal capacitary boundary and Carath\'eodory prime ends}

The notion of the ideal boundary in the terms of prime end was introduced by Carath\'eodory \cite{Cr}. The Cartheodory prime ends represent a compactification of plane domains in the relative distance introduced by Lavrentiev \cite{Lv}. 
(A detailed historical sketch  can be found in \cite{Mi}). We prove that the capacitary boundary is homeomorphic to the Carath\'eodory boundary.

\begin{thm}\label{thm:CapCar} 
Let $\Omega\subset\mathbb R^2$ be a simply connected domain, $\Omega\ne \mathbb R^2$. Then the capacitary boundary $H_{\rho}$ is homeomorphic to the Carath\'eodory boundary $\partial_{C}\Omega$. 
\end{thm}

\begin{proof} The Carath\'eodory boundary $\partial_{C}\Omega$ is homeomorphic to the boundary of the unit disc $\partial \mathbb D$. The capacitary boundary is homeomorphic to the boundary of the unit disc $\partial \mathbb D$ also. Hence the capacitary boundary $H_{\rho}$ is conformally equivalent to the Carath\'eodory boundary $\partial_{C}\Omega$. 
\end{proof} 

On the base of this theorem we give some examples \cite{Eps} of boundary elements $h\in H_{\rho}$ of the conformal capacitary boundary.

\begin{exa}
Let
$$
X=\{(x,y): y=1/3^n\,\, \text{for some}\,\, n\geq 1\,\, \text{and}\,\, -1\leq x\leq 2\}
$$
and
$$
Y=\{(x,y): y=2/3^n\,\, \text{for some}\,\, n\geq 1\,\, \text{and}\,\, -2\leq x\leq 1\}.
$$
Let $\Omega = (-2,2)\times(0,1)\setminus(X\cup Y)$. The boundary element of this domain is 
$h=\{(x,0): -1\leq x\leq 1\}$.
\end{exa}

\begin{exa}
Let $\Omega=\mathbb R^2\setminus K$, where $K$ is given in polar coordinates by 
\begin{multline}
K=\{(r,\theta): \theta=2\pi p\,\,\text{for some integer}\,\, n\geq 1\,\, \text{and some odd integer}\,\,p\,\,\\
\text{such that}\,\, 0<p<2^n, \, 0\leq r\leq 1/2^n\}.
\nonumber
\end{multline}
The boundary element $h\in H_{\rho}$ of this domain at the origin is homeomorphic to a Cantor set. 
\end{exa}

By C.~Caratheodory \cite{Cr} the domain $\Omega$ is called locally connected at boundary points point if and only if every boundary element has trivial realization. Hence we have the following corollary of Theorem \ref{LocCon}:

\begin{cor} 
\label{LocConSim}
Let a simply connected domain $\Omega$ is locally
connected at any point $x\in\partial\Omega$. Then the identical mapping $i:\Omega\to\Omega$ can be extend to a homeomorphism
$\tilde{i_\rho}: \tilde{\Omega}_{\rho}\to\overline{\Omega}$.
\end{cor}

\section{ Strong Luzin Property for the Capacitary Metric and Boundary Values of Sobolev functions}

Recall the notion of the conformal capacity of a set $E\subset \Omega$. Let $\Omega$ be a domain in $\mathbb R^2$ and a compact $F\subset\Omega$. The conformal capacity of the compact $F$ is defined by
$$
\cp(F;\Omega)=\inf\{\|u|L^1_2(\Omega\|^2,\,\,u\geq 1\,\, \text{on}\,\, F, \,\,u\in C_0(\Omega)\}.
$$
By the similar way we can define the conformal capacity of open sets.

For arbitrary set $E\subset\Omega$ we define a inner conformal capacity as 
$$
\underline{\cp}(E;\Omega)=\sup\{\cp(e;\Omega),\,\,e\subset E\subset\Omega,\,\, e\,\,\text{is a compact}\},
$$
and a outer conformal capacity as 
$$
\overline{\cp}(E;\Omega)=\inf\{\cp(U;\Omega),\,\,E\subset U\subset\Omega,\,\, U\,\,\text{is an open set}\}.
$$
A set $E\subset\Omega$ is called conformal capacity measurable, if $\underline{\cp}(E;\Omega)=\overline{\cp}(E;\Omega)$. The value
$$
\cp(E;\Omega)=\underline{\cp}(E;\Omega)=\overline{\cp}(E;\Omega)
$$
is called the conformal capacity of the set $E\subset\Omega$.

The classical Luzin theorem asserts that every measurable function
is {\bf uniformly} continuous if it is restricted to the complement of an
open set of arbitrary small measure. It is reasonable to conjecture
that every function $u\in L^{1}_{2}(\Omega)$ is {\bf uniformly} continuous
if it is restricted to the complement of an open subset of $\Omega\subset R^{2}$
of arbitrary small conformal capacity. Unfortunately this conjecture is
wrong for an arbitrary domain and is correct only under additional
conditions on $\Omega.$ The weak version of The Luzin theorem is correct for the capacity:

\begin{thm} 
\label{thm:WeakLuzin}
(Weak Luzin theorem for $p$-capacity \cite{Maz})
Let $\Omega\subset \mathbb R^2$ be an open set. For any function $u\in L^1_2(\Omega)$ and for any $\varepsilon>0$ there exists an open set $U_{\varepsilon}\in\Omega$, $\cp(U_{\varepsilon};\Omega)<\varepsilon$, such that $u|_{\Omega\setminus U_{\varepsilon}}$ is continuous.
\end{thm}

We discuss here the strong version of the Luzin property for the capacity:

\begin{defn} {\bf Strong Luzin capacitary property}. A domain $\Omega\subset\mathbb R^2$ possesses a strong Luzin capacitary property if for every function $u\in L^{1}_{2}(\Omega)$ and for any $\varepsilon>0$
there exists an open set $U_{\varepsilon}$ of the conformal capacity less
then $\varepsilon$ and such that the restriction of the function $u$ on $\Omega \setminus U_{\varepsilon}$ is uniformly continuous for the conformal capacitary metric.  
\end{defn}

This property looks very restrictive, but, in reality, it is correct for a large set of domains. We prove in this section  that any extension domain possess the Luzin capacitary property and we prove in next section that any quasiconformal homeomorphism preserves the strong Luzin capacitary property. 

Our main motivation for a study of this property is  the following result:

\begin{thm}
\label{thm:Tietz} 
Let a domain $\Omega\subset \mathbb R^{2}$ possesses the strong Luzin capacitary property. Then for any function $u\in L_{2}^{1}(\Omega)$
there exists a quasicontinuous function $\widetilde{u}:\widetilde{\Omega}_{\rho}\to \mathbb R$
defined quasieverywhere on $H_{\rho}$ such that $\tilde{u}|_\Omega=u$. 
\end{thm}

\begin{proof} Because $\Omega$ possesses the strong Luzin capacitary property then for every
$\varepsilon>0$ there exists an open set $U_{\varepsilon}\subset\Omega$
such that $\cp(U_{\varepsilon})<\varepsilon$ and any function $u\in L^1_2(\Omega)$
is uniformly continuous for the conformal capacitary metric on the closed (with respect to $\Omega$)
set $\Omega^{\varepsilon}=\Omega\setminus U_{\varepsilon}$. Consider
the completion $\widetilde{\Omega}^{\varepsilon}$ of the set $\Omega^{\varepsilon}$
in the complete metric space $\left(\widetilde{\Omega}_{\rho},\rho\right)$.
The function $u\in L_{p}^{1}(\Omega)$ will be uniformly continuous
on the metric space $\left(\Omega^{\varepsilon}_{\rho},\rho\right)$.
Hence by the Tietz theorem there exists an extension $\tilde{u}_{\varepsilon}$
of $u$ to $\widetilde{\Omega}^{\varepsilon}$. Put $\widetilde{\Omega}^{0}=\cup_{\varepsilon>0}\widetilde{\Omega}^{\varepsilon}$.
Then the function $u$ possesses an extension $\widetilde{u}$ to
the metric space $\left(\widetilde{\Omega}^{0},\rho\right)$ and $\cp(\widetilde{\Omega}_{\rho}\setminus\widetilde{\Omega}^{0})=0$
because $\Omega_{\varepsilon_{1}}\supset\Omega_{\varepsilon_{2}}$
if ${\varepsilon_{1}}<{\varepsilon_{2}}$. Therefore $\widetilde{u}|_{H_{\rho}}$
defined quasi-everywhere on $H_{\rho}$ and represents the boundary
value of the function $u\in L^1_2(\Omega)$ on the capacitary boundary $H_{\rho}$.
\end{proof}

\begin{rem}
The function $\tilde{u}$ defined quasi-everywhere on $H_{\rho}$
in the following sense. For any $\varepsilon>0$ there exists an open
set $U_{\varepsilon}\subset\Omega$ such that the function $u$ is
uniformly continuous on $\Omega\setminus U_{\varepsilon}$, $\cp(\Omega_{\rho}\setminus U_{\varepsilon})<\varepsilon$,
and the continuous extension of $\widetilde{u}:\Omega\setminus U_{\varepsilon}\to \mathbb R$
to its completion $\left(\widetilde{\Omega\setminus U_{\varepsilon}},\rho\right)$
coincides with $\tilde{u}$ on $H_{\rho}\cap\left(\widetilde{\Omega\setminus U_{\varepsilon}},\rho\right)$.
\end{rem}

Combining the previous Theorem and Theorem \ref{LocConSim} we obtain immediately
\begin{thm}
\label{thm:Tietz1} 
Let a domain $\Omega\subset \mathbb R^{2}$ possesses the strong Luzin capacitary property and be locally connected at any boundary point. Then for any function $u\in L_{2}^{1}(\Omega)$
there exists a quasicontinuous function $\widetilde{u}:\overline{\Omega} \to \mathbb R$
defined quasieverywhere on $\partial \Omega$ such that $\tilde{u}|_\Omega=u$. 
\end{thm}

The strong Luzin capacitary property is valid for the large class of domains, namely extension domains.
The class of extension domains includes domains with smooth or Lipschitz boundaries (see for example \cite{Maz}).

\begin{defn}
A domain $\Omega\subset R^{2}$ is said to be a Sobolev $L_{2}^{1}$
-extension domain if there exists a bounded linear operator $E:L_{2}^{1}(\Omega)\to L_{2}^{1}(R^{n})$
such that for any $u\in L_{2}^{1}(\Omega)$ the condition $E(u)|_{\Omega}=u$
holds. 
\end{defn}

We call the operator $E$  an extension operator. It is known that a simply connected domain $\Omega\subset\mathbb R^2$ is a $L^1_2$-extension domain if and only if $\Omega$ is a quasidisc \cite{GV1}.

Recall that a domain $\Omega\subset\mathbb R^2$ is called a quasidisc if there exists a quasiconformal homeomorphism $\varphi:\mathbb R^2\to \mathbb R^2$ such that $\Omega=\varphi(\mathbb D)$

\begin{thm}
\label{thm:HomForExtension}
If a bounded domain $\Omega\subset\mathbb R^2$ is a $L_{2}^{1}$ -extension domain
then the identity mapping $id:H_{\rho}\to\partial{\Omega}$ is a homeomorphism.
\end{thm}

\begin{proof}
Because $\Omega$ is a $L_{2}^{1}$ -extension domain then there exists
an extension operator 
$$
E:L_{2}^{1}(\Omega)=L_{2}^{1}(R^{n})
$$ 
such
that for any $u\in L_{2}^{1}(\Omega)$ we have $E(u)|_\Omega=u$. Hence
\[
\frac{1}{\left\Vert E\right\Vert }\left\Vert E(u)\right\Vert _{L_{2}^{1}(R^{n})}\leq\left\Vert u\right\Vert _{L_{2}^{1}(\Omega)}\leq\left\Vert E(u)\right\Vert _{L_{2}^{1}(R^{n})}.\]

By the definition of the conformal capacity for any condensor $(F_{0},F_{1};\Omega)$
the following inequality 
\[
\frac{1}{\left\Vert E\right\Vert ^{2}}\cp(F_{0},F_{1};R^{n})\leq \cp(F_{0},F_{1};\Omega)\leq \cp(F_{0},F_{1};R^{n})
\]
holds

So, by the definition of the conformal capacitary metric for any points $x,y\in\Omega$
and any pair $(F,V)$ from the previous inequality follows 
\[
\frac{1}{\left\Vert E\right\Vert^2 }\hat{\rho}_{(F,V)}(x,y)\leq\rho_{(F,V)}(x,y)\leq\hat{\rho}_{(F,V)}(x,y)
\]
where $\hat{\rho}_{(F,V)}(x,y)$ is the conformal capacitary metric in
$R^{2}$ and $\rho_{(F,V)}(x,y)$ is the conformal capacitary metric
in $\Omega$. It means that the metric $\rho_{(F,V)}(x,y)$ is equivalent
to the metric $\hat{\rho}_{(F,V)}(x,y)$ on $\Omega$. By Theorem~\ref{thm:CoinTopol}
the topology induced by the metric $\hat{\rho}_{(F,V)}(x,y)$
on $\mathbb R^2$ coincides with the Euclidean topology and so the topology of $H_{\rho}$ coincides with the Euclidean topology of $\partial\Omega$. Because metrics
$\rho_{(F,V)}(x,y)$ and $\hat{\rho}_{(F,V)}(x,y)$ are equivalent
on $\Omega$ the theorem proved.
\end{proof}

\begin{thm}
\label{thm:ExtLuzin}
Let $\Omega\subset\mathbb R^2$ be a bounded $L_{2}^{1}$ -extension domain. Then $\Omega$  possesses the strong Luzin capacitary property.
\end{thm}

\begin{proof}
Choose arbitrarily a function $u\in L_{2}^{1}$. Because $\Omega$
is an extension domain there exists an extension $\widehat{u}\in L_{2}^{1}(R^{2})$
of $u$. By the Theorem~\ref{thm:WeakLuzin} for any $\varepsilon>0$
there exist such open set $U_{\varepsilon}\in R^{2}$ of conformal capacity less then
$\varepsilon$ such that the function $\widehat{u}$ is continuous on $R^{2}\setminus U_{\varepsilon}$.
Because the domain $\Omega$ is bounded the function $\hat{u}|_{\bar{\Omega}\setminus U_{\varepsilon}}$
is uniformly continuous. Hence the function $u$ is uniformly continuous
on $\Omega\setminus U_{\varepsilon}$ . By monotonicity of conformal capacity $\cp(U_{\varepsilon}\cap\Omega)<\cp(U_{\varepsilon})<\varepsilon$.
By the previous Theorem~\ref{thm:HomForExtension} the function $u$ is uniformly continuous
for $p$-capacitary metric in $\Omega\setminus U_{\varepsilon}$ also.
\end{proof}

Combining Theorem \ref{thm:Tietz}, Theorem \ref{thm:ExtLuzin} and \ref{thm:HomForExtension} we obtain
\begin{thm}
\label{thm:TietzExtension} 
Let a domain $\Omega\subset \mathbb R^{2}$ be a bounded $L_{2}^{1}$ -extension domain. 
Then for any function $u\in L_{2}^{1}(\Omega)$
there exists a quasicontinuous function $\widetilde{u}:\overline{\Omega} \to \mathbb R$
defined quasieverywhere on $\partial \Omega$ such that $\tilde{u}|_\Omega=u$. 
\end{thm}

 Theorem \ref{thm:Tietz}, Theorem \ref{thm:ExtLuzin} and \ref{thm:HomForExtension} can be easily extended to a more flexible class of so-called quasi-extension domains:

\begin{defn}
A domain $\Omega\subset \mathbb R^{2}$ is said to be a Sobolev $L_{2}^{1}$
-quasi-extension domain if for any $\varepsilon>0$ there exist such
open set $U_{\varepsilon}$ of conformal capacity less then $\varepsilon$ that $\Omega\setminus\bar{U_{\varepsilon}}$
is a $L_{2}^{1}$ -extension domain. 
\end{defn}

Typical examples of such domains are domains with boundary singularities
of conformal capacity zero. 

\begin{thm}
If a bounded domain $\Omega$ is a $L_{2}^{1}$ quasi-extension domain
then the identity mapping $id:\partial{\Omega}\to H_{\rho}$ is a homeomorphism.
\end{thm}

\begin{proof}
Follows directly from Theorem~\ref{thm:HomForExtension} and the countable subadditivity
of capacity.
\end{proof}

\begin{thm}
Let $\Omega\subset\mathbb R^2$ be a bounded $L_{2}^{1}$ quasi-extension domain. Then $\Omega$  possesses the strong Luzin capacitary property.
\end{thm}

\begin{proof}
Follows directly from Theorem~\ref{thm:HomForExtension} and the countable subadditivity
of capacity.
\end{proof}

\begin{thm}
Let a domain $\Omega\subset \mathbb R^{2}$ be a bounded $L_{2}^{1}$ quasi-extension domain. 
Then for any function $u\in L_{2}^{1}(\Omega)$
there exists a quasicontinuous function $\widetilde{u}:\overline{\Omega} \to \mathbb R$
defined quasieverywhere on $\partial \Omega$ such that $\tilde{u}|_\Omega=u$. 
\end{thm}

\section{Boundary Values of Sobolev Functions for Simply Connected Domains}

Using the Riemann Mapping Theorem we prove that any simply connected domain possess the strong Luzin capacitary property, that permits to extend main results to any simply connected domain.

The unit disk $\mathbb D(0,1)\subset\mathbb R^2$ is the $L^1_2$-extension domain and possess the strong Luzin capacitary property.  Remember that the conformal capacity of condensors is a quasi-invariant for quasiconformal homeomorphisms $\varphi:\Omega \to \Omega'$ between two plane domains $\Omega$ and $\Omega'$. Hence the conformal capacitary metric is also a quasi-invariant for quasiconformal homeomorphisms. Moreover, from this remark immediately follows

\begin{prop} (\cite{GV}
\label{QuasiInv}
Any quasiconformal homeomorphism $\varphi:\Omega \to \Omega'$ between two plane domains $\Omega$ and $\Omega'$ induces a quasi-isometry of $\widetilde{\Omega}_{\rho}$ and $\widetilde{\Omega'}_{\rho}$.
\end{prop}

\begin{cor}
Let $\varphi: \mathbb D \to \Omega$ be a quasiconformal homeomorphism of the unit disc $\mathbb D$ onto a domain $\Omega \in R^2$. Then $\Omega$ possesses the strong Luzin capacitary property.
\end{cor}

\begin{proof}
Choose a function $u\in L^1_2(\Omega)$. Because $\varphi:\mathbb D \to \Omega$ is a quasiconformal homeomorphism, then the composition $v:=u \circ \varphi$ belongs to $L^1_2(D)$ (see, for example \cite{GR}). Because $\mathbb D$ possesses the Luzin capacitary property, then for any $\varepsilon >0$ there exist an open set $V_{\varepsilon}$ of the conformal capacity less then $\varepsilon$ such that $v|_{D \setminus V_{\varepsilon}}$ is uniformly continuous. The conformal capacity is a quasiinvariant for a quasiconformal homeomorphism $\varphi$. It means that there exist a constant $Q$ which depends only on the quasiconformal distortion of $\varphi$ and such that the conformal capacity of $U_{\varepsilon}:=\varphi(V_{\varepsilon})$ is less then $Q \varepsilon$. By the previous proposition $\varphi^{-1}$ induces a quasiisometry of $\widetilde{\Omega}_{\rho}$ and $\widetilde{D}_{\rho}$. Therefore $u=v \circ \varphi^{-1}$ is uniformly continuous on $G \setminus U_{\varepsilon}$. We proved that $\Omega$ possesses the strong Luzin capacitary property
 \end{proof}

From the previous proposition and Theorem \ref{thm:Tietz} immediately follows 
 
\begin{thm} 
\label{MainTh}
Let $\Omega\subset\mathbb R^2$ be a simply connected domain, $\Omega\ne\mathbb R^2$. Then for any function $u\in L_{2}^{1}(\Omega)$
there exists a quasicontinuous function $\widetilde{u}:\widetilde{\Omega}_{\rho}\to \mathbb R$
defined quasi-everywhere on the capacitary boundary $H_{\rho}$ such that $\tilde{u}|_{\Omega}=u$. 
\end{thm}

And its version for simply connected domains locally connected at any boundary point

\begin{thm} 
\label{MainTh1}
Let $\Omega\subset\mathbb R^2$, $\Omega\ne\mathbb R^2$ be a simply connected domain locally connected at any boundary point. Then for any function $u\in L_{2}^{1}(\Omega)$
there exists a quasicontinuous function $\widetilde{u}:\overline{\Omega} \to \mathbb R$
defined quasi-everywhere on the boundary $\partial \Omega$ such that $\tilde{u}|_{\Omega}=u$. 
\end{thm}

For readers convenience we repeat some basic facts about quasidiscs.

\begin{defn}
A domain $\Omega$ is called a $K$-quasidisc if it is an image of the
unit disc $\mathbb{D}$ of a $K$-quasiconformal homeomorphism of
the plane onto itself.
\end{defn}

 It is well known that the boundary of any $K$-quasidisc $\Omega$
admits a $K^{2}$-quasiconformal reflections and thus, for example,
any conformal homeomorphism $\varphi:\mathbb{D}\to\Omega$ can be
extended to a $K^{2}$-quasiconformal homeomorphism of the hole plane
to itself.

Boundaries of quasidisc are called quasicircles. It is known that there are quasicircles for which no segment has finite length. The Hausdorff dimension of quasicircles was first investigated by Gehring and Vaisala (1973) \cite{GV73}, who proved that it can take all values in the interval $[1,2)$. S. Smirnov proved recently \cite{Smi10} that the Hausdorff dimension of
any $K$-quasicircle is at most $1+k^2$, where $k = (K-1)/(K +1)$.

Ahlfors's 3-point condition \cite{Ahl63} gives
a complete geometric characterization: a Jordan curve $\gamma$ in the plane is a quasicircle
if and only if for each two points $a, b$ on $\gamma$ the (smaller) arc between them has
diameter comparable to $|a-b|$. This condition is easily checked for the snowflake.
On the other hand, every quasicircle can be obtained by an explicit snowflake-type
construction (see \cite{Roh01}). 

Because any quasidisc is an $L_2^1$-extension domain we can reformulate previous results in the terms of quasidiscs.

\begin{prop}
\label{prop:HomForExtension}
Let a domain $\Omega\subset\mathbb R^2$ is a quasidisc. Then the identity mapping $id:H_{\rho}\to\partial{\Omega}$ is a homeomorphism.
\end{prop}

\begin{prop}
\label{prop:ExtLuzin}
Let a domain $\Omega\subset\mathbb R^2$ is a quasidisc. Then $\Omega$  possesses the strong Luzin capacitary property.
\end{prop}

\section {Historical Sketch and Conclusions}

The concept of the ideal boundaries is common for the geometry and the analysis.
The Poincare disc is a model of the hyperbolic plane that provides a geometrical
realization of the ideal boundary of the hyperbolic plane with help
of a conformal homeomorphism. 

By the Riemann Mapping Theorem any simply connected plane domain  $\Omega\ne\mathbb R^2$ is conformally equivalent to the unit disc. However the boundary
behavior of plane conformal homeomorphisms can not be described in
terms of Euclidean boundaries but it can be described  in terms of  ideal boundary
elements (prime ends) that  was introduced by C.~Caratheodory. By the Caratheodory Theorem  any conformal homeomorphism $\varphi: \mathbb{D}\to \Omega$  induces one to one correspondence of prime ends. 

M.~A.~Lavrentiev \cite{Lv} introduced a metric (a relative distance)
for prime ends. G.~D.~Suvorov \cite{Su} constructed a counterexample that demonstrates
an absence of the triangle inequality for the Lavrentiev relative
distance and proposed a more accurate concept of relative distance that support the triangle inequality. In terms of this metric the Cartheodory prime ends are a geometric representation of "ideal" compactification "boundary points". There exists a number of different
conformally invariant intrinsic metrics. A detailed survey can be
found in the paper of V.~M.~Miklyukov \cite{Mi}.

For dimension more than two by the Liuoville theorem the class of conformal homeomorphisms
coincides with the Mobius transformations. Even for quasiconformal homeomorphisms nothing similar to the Riemann mapping theorem is not correct.

By our opinion two main constructions of  a quasiconformally invariant
{}``ideal'' boundary were proposed. The first one was in the spirit of Banach algebras. 
Recall that Royden algebra $\mathbb{R}(\Omega)$ is a  quasiconformal invariant by M.~Nakai \cite{Na} for dimension two and by L.~G.~Lewis \cite{Le} for arbitrary dimension. As any
Banach algebra the Royden algebra produces a compactification of $\Omega$
and any quasiconformal homeomorphisms induces a homeomorphism
of such compactifications.

The second one is so-called capacitary boundary proposed by
V.~Gol'd\-shtein and S.~K.~Vodop'janov \cite{GV}. Its construction is based on a notion of the conformal capacity. Remember that the conformal capacity is a quasiinvariant
of quasiconformal homeomorphisms. By \cite{GV} quasiconformal homeomorphism
can be extended to a homeomorphism of domains with capacitary boundaries.

The Royden compactification does not coincide with the Caratheodory
compactification. The "ideal" elements  of the capacitary boundary  are  Caratheodory
prime ends. 

Necessary and sufficient condition for existence of continuous traces of  $L^1_P(\Omega)$, $p>2$ were obtained by Shvartsman \cite{Sh} in terms of quasi-hyperbolic metrics.

\end{document}